\documentclass[12pt, reqno]{amsart}
\usepackage{amssymb,amsmath,amsthm,cancel,hyperref,color, enumerate}
\usepackage[pdftex]{graphicx}
\hypersetup{colorlinks=true,linkcolor=blue,citecolor=magenta}
\usepackage[all]{xy}
\usepackage{a4wide}
\usepackage{amssymb}
\usepackage{amsthm}
\usepackage{amsmath}
\usepackage{tablefootnote}
\usepackage{amscd}
\usepackage{verbatim}
\usepackage{threeparttable}
\usepackage[OT2,T1]{fontenc}
\DeclareSymbolFont{cyrletters}{OT2}{wncyr}{m}{n}
\DeclareMathSymbol{\Sha}{\mathalpha}{cyrletters}{"58}
\textheight8.5in \textwidth6.75in \numberwithin{equation}{section}

\theoremstyle{plain}
\newtheorem{theorem}{Theorem}[section]
\newtheorem{corollary}[theorem]{Corollary}
\newtheorem{lemma}[theorem]{Lemma}

\theoremstyle{definition}

\newtheorem*{example}{Example}

\theoremstyle{remark}
\newtheorem{remark}{Remark}

\renewenvironment{proof}[1][Proof]{\begin{trivlist}
\item[\hskip \labelsep {\bfseries #1:}]}{\qed\end{trivlist}}

\newcommand{\Q}{\mathbb{Q}}
\newcommand{\Z}{\mathbb{Z}}
\newcommand{\N}{\mathbb{N}}
\newcommand{\C}{\mathbb{C}}
\newcommand{\h}{\mathbb{H}}
\renewcommand{\H}{\mathbb{H}}

\newcommand{\ord}{\textrm ord}

\newcommand{\calH}{\mathcal{H}}

\newcommand{\calO}{\mathcal{O}}

\newcommand{\pa}[1]{\left( {#1} \right)}

\newcommand{\tr}{\operatorname{tr}}

\newcommand{\Sl}{\operatorname{SL}}

\newcommand{\Mp}{\operatorname{Mp}}

\newcommand{\Aut}{\operatorname{Aut}}
\newcommand{\Mat}{\operatorname{Mat}}

\newcommand{\Disc}{\operatorname{Disc}}

\newcommand{\SL}{{\text {\rm SL}}}
\newcommand{\G}{\Gamma}

\newcommand{\e}{\mathfrak{e}}
\newcommand{\smallabcd}{\left(\begin{smallmatrix}a & b \\ c & d\end{smallmatrix}\right)}

\newcommand{\smallTmatrix}{\left(\begin{smallmatrix}1 & 1 \\ 0 & 1\end{smallmatrix}\right)}
\newcommand{\smallSmatrix}{\left(\begin{smallmatrix}0 & -1 \\ 1 & 0\end{smallmatrix}\right)}

\begin{document}

\newcounter{itemcounter}
\newcounter{itemcounter2}

\title[Classical and Umbral Moonshine: Connections and $p$-adic Properties]
{Classical and Umbral Moonshine: Connections and $p$-adic Properties}
    \author{Ken Ono, Larry Rolen, Sarah Trebat-Leder}

\address{Mathematics Institute, University of Cologne, Gyrhofstr. 8b, 50931 Cologne, Germany}
\email{lrolen@math.uni-koeln.de}

\address{Department of Mathematics, Emory University, Emory, Atlanta, GA 30322}
\email{ono@mathcs.emory.edu}

 \address{Department of Mathematics, Emory University, Emory, Atlanta, GA 30322}
\email{strebat@emory.edu}

\thanks{The first author thanks the NSF and the Asa Griggs Candler Fund for their generous support. The second author thanks the University of Cologne and the DFG for their generous support via the University of Cologne postdoc grant DFG Grant D-72133-G-403-151001011, funded under the Institutional Strategy of the University of Cologne within the German Excellence Initiative. The third author thanks the NSF for its support. The authors began jointly discussing this work at the mock modular forms, moonshine, and string theory conference in Stony Brook, August 2013 and are grateful for the good hospitality and excellent conference. The authors are also grateful to John Duncan and Jeffery Harvey for useful comments which improved the quality of exposition.}

\begin{abstract}
The classical theory of \textit{monstrous moonshine} describes the unexpected connection between the representation theory of the monster group $M$, the largest of the sporadic simple groups, and certain modular functions, called Hauptmoduln.  In particular, the $n$-th Fourier coefficient of Klein's $j$-function is the dimension of the grade $n$ part of a special infinite dimensional representation $V^{\natural}$ of the monster group.  More generally the coefficients of Hauptmoduln are graded traces $T_g$ of $g \in M$ acting on $V^{\natural}$.  Similar phenomena have been shown to hold for the Mathieu group $M_{24}$, but instead of modular functions, \textit{mock modular forms} must be used. This has been conjecturally generalized even further, to \textit{umbral moonshine}, which associates to each of the 23 Niemeier lattices a finite group, infinite dimensional representation, and mock modular form.  We use \textit{generalized Borcherds products} to relate monstrous moonshine and umbral moonshine. Namely, we use mock modular forms from umbral moonshine to construct via generalized Borcherds products rational functions of the Hauptmoduln $T_g$ from monstrous moonshine.  This allows us to associate to each pure $A$-type Niemeier lattice a conjugacy class $g$ of the monster group, and gives rise to identities relating dimensions of representations from umbral moonshine to values of $T_g$.  We also show that the logarithmic derivatives of the Borcherds products are $p$-adic modular forms for certain primes $p$ and describe some of the resulting properties of their coefficients modulo $p$. 
\end{abstract}

\maketitle

\noindent

\section{Introduction}

\textit{Monstrous moonshine} begins with the surprising connection between the coefficients of the modular function 
$$J(\tau):=j(\tau)-744 = \frac{(1 + 240 \sum_{n = 1}^\infty \sum_{d \mid n}d^3 q^n)^3}{q \prod_{n = 1}^\infty (1 - q^n)^{24}}-744 =  \frac{1}{q} + 196884 q + 21493760 q^2 + \dots \; \;$$
 and the representation theory of the monster group $M$, which is the largest of the simple sporadic groups. Here $q:=e^{2\pi i\tau}$ and $\tau\in\H:=\{z\in\C\colon\Im z>0\}$. McKay noticed that $196884$, the $q^1$ coefficient of $J(\tau)$, can be expressed as a linear combination of dimensions of irreducible representations of the monster group $M$.  Namely, 
$$196884 = 196883 + 1.$$
McKay saw that the same was true for other Fourier coefficients of $J(\tau)$.  For example, 
$$21493760 = 21296876 + 196883 + 1.$$  In \cite{Thompson:1979tm}, McKay and Thompson conjectured that the $n$-th Fourier coefficient of $J(\tau)$is the dimension of the grade $n$ part of a special infinite-dimensional graded representation $V^{\natural}$ of $M$.  

This was later expanded into the full monstrous moonshine conjecture by Thompson, Conway, and Norton \cite{Conway:1979ul, Thompson:1979ev}.  Since the graded dimension is just the graded trace of the identity element, they looked at the graded traces $T_g(\tau)$ of nontrivial elements $g$ of $M$ acting on $V^{\natural}$ and conjectured that they were all expansions of principal moduli, or Hauptmoduln, for certain genus zero congruence groups $\Gamma_g$ commensurable with $\SL_2(\Z)$.  Note that these $T_g$ are constant on each of the 194 conjugacy classes of $M$, and therefore are class functions, which automatically have coefficients which are $\C$-linear combinations of irreducible characters of $M$.  Part of the task of proving monstrous moonshine was showing that they were in fact $\Z_{\geq 0}$-linear combinations.  

By way of computer calculation, Atkin, Fong, and Smith \cite{Smith:1985hf} verified the existence of a virtual representation of M. Then using vertex-operator theory, Frenkel, Lepowsky, and Meurman \cite{Frenkel:1984uj} finally constructed a representation $V^{\natural}$ of $M$ thereby providing a beautiful algebraic explanation for the original numerical observations of McKay and Thompson. Borcherds \cite{Borcherds:1986ui} further developed the theory of vertex-operator algebras, which he then used in \cite{Borcherds:1992ua} to prove the full conjectures as given by Conway and Norton.

Monstrous moonshine provides an example of coefficients of modular functions enjoying distinguished properties.  Moreover, their values at Heegner points have also been considered important.
A \textit{Heegner point} $\tau$ of discrimant $d < 0$ is a complex number of the form $\tau = \frac{-b \pm \sqrt{b^2 - 4 ac}}{2a}$ with $a, b, c \in \Z$, $\gcd(a, b, c) = 1$, and $d = b^2 - 4 a c$.  The values of principal moduli at such points are called \textit{singular moduli}.  As an example of their importance, it is a classical fact that the singular moduli of $j(\tau)$ generate Hilbert class fields of imaginary quadratic fields. Moreover, the other McKay-Thompson series arising in monstrous moonshine satisfy analogous properties \cite{ChenYui}.  It is natural to ask what other interesting properties the values of the Hauptmoduln $T_g(\tau)$ could possess.  We show that some of these values are related to another kind of moonshine, called \textit{umbral moonshine}. 

Recently, it was shown that phenomena similar to monstrous moonshine occur for other $q$-series and groups.  In particular, the Mathieu group $M_{24}$ exhibits moonshine \cite{Eguchi:2011ux, Gannon:2012uv}, with the role of the $j$-invariant played by a \textit{mock modular form} of weight $1/2$, denoted $H^{(2)}(\tau)$.  A mock modular form is the holomorphic part of a \textit{harmonic weak Maass form}.  Cheng, Duncan, and Harvey conjecture in \cite{Cheng:2013vr} that this is a special case of a more general phenomenon, which they call \textit{umbral moonshine}.  For each of the 23 Niemeier lattices $X$ they associate a vector-valued mock modular form $H^X(\tau)$, a group $G^X$, and an infinite-dimensional graded  representation $K^X$ of $G^X$ such that the Fourier coefficients of $H^X$ encode the dimensions of the graded components of $K^X$. 

In particular, if $c^X(n, h)$ is the $n$-th Fourier coefficient of the $h$-th component of $H^X$, then 
\begin{equation} 
\label{coeff-dim}
c^X(n, h) 
 = \left\{
\begin{array}{l l}
a^X \dim_{K^X_{h, -D/4m}} & \text{if $n = -D/4m$ where $D \in \Z, D = h^2 \pmod{4m}$,}\\
0 & \text{otherwise,}  \\
\end{array}
\right.
\end{equation}
where $a^X \in \{ 1, 1/3\}$ and $$K^X = \bigoplus_{h \pmod{2m}} \bigoplus_{\substack{D \in \Z \\ D = h^2 \pmod{4m}}} K^X_{h, -D/4m}.$$
For more information on umbral moonshine see Section~\ref{sec:umbral} and for a definition of $H^X$ see Section~\ref{sec:definitions}. 

Using \textit{generalized Borcherds products} (see \cite{Bruinier:2010ff}), we describe a connection between the mock modular forms $H^X(\tau)$ of umbral moonshine and the McKay-Thompson series $T_g(\tau)$ of monstrous moonshine. \textit{Generalized Borcherds products} are a method to produce modular functions as infinite products of rational functions whose exponents come from the coefficients of mock modular forms, and they can be viewed as generalizations of the automorphic
products in Theorem 13.3 of \cite{Borcherds}. 

We focus on the Niemeier lattices $X$ whose root systems are of pure $A$-type according to the ADE classification.  They are listed in Table~\ref{intro-umbral}, along with their Coxeter numbers $m(X)$ and the notation we will use for the mock modular form $H^X$.  

\begin{table}[h!]
\caption{Pure $A$-type root systems}

\[ \arraycolsep=10pt\def\arraystretch{1.5}
\begin{array}{| c | c | c |}
\hline
\text{Root System $X$} & \text{Coxeter Number $m(X)$} &\text{mock modular Form $H^X$} \\
\hline \hline
A_1^{24} & 2 & H^{(2)}(\tau) \\
\hline
 A_2^{12} & 3 & H^{(3)}(\tau) \\
\hline
 A_3^{8} & 4 & H^{(4)}(\tau)\\
\hline
A_4^{6} & 5 & H^{(5)}(\tau)\\
\hline
A_6^{4}& 7 & H^{(7)}(\tau)\\
\hline
A_8^{3} & 9 & H^{(9)}(\tau)\\
\hline
A_{12}^{2} & 13 & H^{(13)}(\tau)\\
\hline
 A_{24}^{1} & 25 & H^{(25)}(\tau)\\
\hline
\end{array}
\]
\label{intro-umbral}
\end{table}

Table~\ref{monster} gives the monstrous moonshine dictionary for the conjugacy classes $g$ which correspond to pure $A$-type cases of umbral moonshine\footnote{The case $X=A_{24}$ corresponds to $g(X) = (25Z)$, which is what Conway and Norton call a ``ghost element''.  This means that $\Gamma_0(25)$ is the only genus zero $\Gamma_0(N)$ that does not correspond to a conjugacy class of the monster group. The parentheses are used to indicate a ghost element.}.  Note that $\eta(\tau)$ is the \textit{Dedekind eta function}, defined by $$\eta(\tau) := q^{1/24}\prod_{n = 1}^\infty (1 - q^n).$$  All of our Hauptmoduln are normalized so that they have the form $q^{-1} + O(q)$, which is why all of the $\eta$-quotients in the table have a constant added to them.  

\begin{table}[h!]
\caption{The dictionary of monstrous moonshine}
\[
\arraycolsep=10pt\def\arraystretch{1.5}
\begin{array}{| c | c | c |}
\hline
\text{Monster Conjugacy Class $g$} & \text{Congruence Subgroup $\Gamma_g$} & \text{McKay-Thomspon Series $T_g(\tau)$} \\
\hline
\hline
2B & \Gamma_0(2) & \eta(\tau)^{24}/\eta(2\tau)^{24} + 24 \\
\hline
3B  &\Gamma_0(3) & \eta(\tau)^{12}/\eta(3\tau)^{12}  + 12\\
\hline
4C & \Gamma_0(4) & \eta(\tau)^{8}/\eta(4\tau)^{8} + 8  \\
\hline
5B & \Gamma_0(5) & \eta(\tau)^{6}/\eta(5\tau)^{6}  + 6 \\
\hline
7B & \Gamma_0(7) & \eta(\tau)^{4}/\eta(7\tau)^{4} + 4  \\
\hline
9B & \Gamma_0(9) & \eta(\tau)^{3}/\eta(9\tau)^{3} + 3  \\
\hline
13B & \Gamma_0(13) & \eta(\tau)^{2}/\eta(13\tau)^{2} + 2\\
\hline
(25 Z) & \Gamma_0(25) & \eta(\tau)/\eta(25\tau) + 1\\
\hline
\end{array}
\]
\label{monster}
\end{table}

There is an evident correspondence between the pure $A$-type lattices $X$ in Table~\ref{intro-umbral} and the conjugacy classes $g$ in Table~\ref{monster}.  We give this correspondence in Table~\ref{correspondence}.

We show that for a pure $A$-type Niemeier lattice $X$ and its corresponding conjugacy class $g := g(X)$, the ``Galois (twisted) traces'' of the CM values of the McKay-Thompson series $T_g(\tau)$ are the coefficients of the mock modular form $H^X$.   To more precisely state this, we set up the following notation. 

Let $X$ be a pure $A$-type Niemeier lattice with Coxeter number $m := m(X)$ and corresponding conjugacy class $g := g(X)$.  We call a pair $(\Delta, r)$ \textit{admissible} if $\Delta$ is a negative fundamental discriminant and $r^2 \equiv \Delta \pmod{4m}$.  We also let $e(a):=e^{2\pi i a}$. 

\begin{theorem}\label{mainthm}
Let $c^+(n, h)$ be the $n$-th Fourier coefficient of the $h$-th component of $H^X$.  
Let $(\Delta, r)$ be an admissible pair for $X$.  Then the twisted generalized Borcherds product
$$\Psi_{\Delta, r}(\tau, H^X):= \prod_{n = 1}^\infty P_\Delta(q^n)^{c^+\pa{\frac{|\Delta| n^2}{4 m}, \frac{rn}{2m}}}, $$
where 
$$P_\Delta(x) := \prod_{b \in \Z/|\Delta|\Z} [1 - e(b/\Delta)x]^{\pa{\frac{\Delta}{b}}}$$
is a rational function in $T_g(\tau)$ with a discriminant $\Delta$ Heegner divisor.    
\end{theorem}
\begin{remark}
We consider only the pure $A$-type cases, because these are the ones for which the harmonic Maass form transforms under the Weil representation.  See Section~\ref{sec:definitions} for more information.  
\end{remark}

The next result gives a precise description of the rational functions in Theorem~\ref{mainthm}.  In particular, it gives a ``twisted'' trace function for the values of $T_g$ at points in the divisor and the coefficients $c^+$ of the mock modular forms $H^X$.  It is often the case that coefficients of automorphic forms can be expressed in terms of singular moduli (see e.g., \cite{ Bringmann:2007fk,Bruinier:2006iy, Duke:2011vb, Zagier:2002tp}).

\begin{table}[h!]
\caption{Correspondence Between Umbral and Monstrous Moonshine}
\[ \arraycolsep=10pt\def\arraystretch{1.5}
\begin{array}{| c | c |}
\hline
\text{Root System $X$} & \text{Conjugacy Class $g(X)$}  \\
\hline \hline
A_1^{24} & 2B \\
\hline
 A_2^{12} & 3B \\
\hline
 A_3^{8} & 4C \\
\hline
A_4^{6} & 5B\\
\hline
A_6^{4}& 7B\\
\hline
A_8^{3} & 9B\\
\hline
A_{12}^{2} & 13B\\
\hline
 A_{24}^{1} & (25Z) \\
\hline
\end{array}
\]
\label{correspondence}
\end{table}

\begin{corollary}
\label{twisted_trace}
By Theorem \ref{mainthm}, we can write $$\Psi_{\Delta, r}(\tau, H^X) = \prod_i \pa{T_g(\tau) - T_g(\alpha_i)}^{\gamma_i}$$ for some discriminant $\Delta$ Heegner points $\alpha_i$.  Then we have that $$c^+\pa{\frac{|\Delta|}{4 m}, \frac{r}{2m}}  = \frac{1}{\epsilon_\Delta} \sum_i \gamma_i \cdot T_g(\alpha_i), $$
where $$\epsilon_\Delta =  \sum_{b \in \Z/|\Delta|\Z} e(b/\Delta) \cdot \pa{\frac{\Delta}{b}}.$$
\end{corollary}

\begin{remark}
Assuming the umbral moonshine conjecture, the previous corollary implies the following ``degree'' formula in traces of singular moduli for classical moonshine functions: 
\begin{equation} 
\frac{1}{\epsilon_\Delta} \sum_i \gamma_i \cdot T_g(\alpha_i) = c^+\pa{\frac{|\Delta|}{4 m}, \frac{r}{2m}} = a^X \dim_{K^X_{r, |\Delta|/4m}}.
\end{equation}

In the case where $m = 2$, the relationship between the coefficients of the mock-modular form and the dimensions of the graded components of the representation has been proven by Gannon \cite{Gannon:2012uv}, and so our work implies the following:
\begin{equation} 
\frac{1}{\epsilon_\Delta} \sum_i \gamma_i \cdot T_{2B}(\alpha_i) = c^+\pa{\frac{|\Delta|}{8}, \frac{r}{4}} = \dim_{K^{(2)}_{r, |\Delta|/8}}.
\end{equation}
\end{remark}

\begin{example}
\label{borcherds-example}
 Let $X = A_1^{24}$, so $m(X) = 2$ and $g(X) = 2B$.  Then the corresponding McKay-Thompson series is $$T_g(\tau) = \frac{\eta(\tau)^{24}}{\eta(2 \tau)^{24}} + 24 = \frac{1}{q} + 276 q + \dots.$$  We pick the admissible pair $(\Delta, r) = (-7, 1)$.  In Section~\ref{relating}, we will show that
\begin{eqnarray*}
\Psi_{\Delta, r}(\tau, H^X)&=& \frac{\pa{T_{g}(\tau) - T_g(\alpha_1)}^2}{\pa{T_{g}(\tau) - T_g(\alpha_2)}^2} = \frac{\pa{T_{g}(\tau) - \frac{1 - 45 \sqrt{-7}}{2}}^2}{\pa{T_{g}(\tau) - \frac{1 + 45 \sqrt{-7}}{2}}^2}\\
&=& 1 + 90 \sqrt{-7} q + (28350 + 45 \sqrt{-7}) q^2 + \dots,
\end{eqnarray*}
where $\alpha_1 := \frac{-1 + \sqrt{-7}}{4}$ and $\alpha_2 := \frac{1 + \sqrt{-7}}{4}$.  Note that $T_g(\alpha_1)$ and $T_g(\alpha_2)$ are algebraic integers of degree 2 which form a full set of conjugates.  Their twisted trace is $$2[T_g(\alpha_1) - T_g(\alpha_2)] = -90\sqrt{-7},$$ which matches the $q^1$ Fourier coefficient above.  To check Corollary~\ref{twisted_trace}, we note that $$\epsilon_\Delta = \sum_{b \in \Z/7\Z}e(-b/7) \cdot \pa{\frac{-7}{b}} = -\sqrt{-7}$$ and
$$\frac{1}{\epsilon_\Delta} \sum_i \gamma_i T_{g}(\alpha_i) = 90 = c^+\pa{7/8, 1/4} = \dim_{K^{(2)}_{1, 7/8}}.$$
\end{example}

\begin{example}
As a second example, again consider $X = A_1^{24}$, so $m(X) = 2$ and $g(X) = 2B$. We pick the admissible pair $(\Delta, r) = (-15, 1)$.  Let $\rho_1, \rho_2, \rho_3,\rho_4$ be the roots of $$x^4-47 x^3+192489 x^2-9012848 x+122529840,$$ with $\rho_1, \rho_2$ having positive imaginary parts.  Then 
$$\Psi_{-15, 1} = \frac{(T_g(\tau) - \rho_1)^2 (T_g(\tau) - \rho_2)^2}{(T_g(\tau) - \rho_3)^2 (T_g(\tau) - \rho_4)^2}.$$ We get that $$\epsilon_{-15} = \sqrt{-15},$$ and  
$$\frac{1}{\epsilon_\Delta} \sum_i \gamma_i T_{g}(\alpha_i) = 462 = c^+\pa{15/8, 1/4} = \dim_{K^{(2)}_{1, 15/8}}.$$
\end{example}

In view of this correspondence, it is clear that the mock modular forms of umbral moonshine have important properties.  The congruence properties of their coefficients have just begun to be studied.  For example, \cite{Creutzig:2012vx} examines the parity of the coefficients of the McKay-Thompson series for Mathieu moonshine in relation to a certain conjecture in \cite{Cheng:2012ue}, which in our case corresponds to $X = A_{1}^{24}$. Congruences modulo higher primes were also considered in \cite{Miezaki:ux}.

Let $\Theta:= q \frac{d}{dq} = \frac{1}{2 \pi i} \frac{d}{d\tau}$.  Given the product expansion of a generalized Borcherds product, it is natural to consider its logarithmic derivative. It turns out that this logarithmic derivative has nice arithmetic properties. This idea was also used in \cite{Bruinier:2010ff} and \cite{Ono:2010ws}.

\begin{theorem}
\label{mero-modform}
Fix a pure $A$-type Niemeier lattice $X$ with Coxeter number $m$.  Let $(\Delta, r)$ be an admissible pair. Consider the logarithmic derivative  $$f_{\Delta, r}(\tau)= \sqrt{\Delta}\sum a_{\Delta, r}(n)q^n :=  \sqrt{\Delta} \sum_{n}  \sum_{ij = n} i c^+\pa{\frac{|\Delta|i^2}{4m}, \frac{ri}{2m}}\pa{\frac{\Delta}{j}}q^n$$ of $\Psi_{\Delta, r}(\tau) = \Psi_{\Delta, r}(\tau, H^X)$.  Then $f_{\Delta, r}(\tau)$ is a meromorphic weight $2$ modular form.
\end{theorem}
 
When $p$ is inert or ramified in $\Q(\sqrt{\Delta})$, it turns out that $f_{\Delta, r}(\tau)$ is more than just a meromorphic modular form; it is a $p$-adic modular form. Essentially, a $p$-adic modular form is a $q$-series which is congruent modulo any power of $p$ to a holomorphic modular form; we refer the reader to Section \ref{sec:padic-def} for the definition.

\begin{theorem}
\label{padic}
Let $X$ be a pure $A$-type Niemeier lattice with Coxeter number $m$.  Let $(\Delta, r)$ be admissible and suppose $p$ is inert or ramified in $\Q(\sqrt{\Delta})$. Then $f_{\Delta, r}$ is a $p$-adic modular form of weight $2$.  
\end{theorem}

We will use this result to study the $p$-divisibility of the coefficients $a_{\Delta, r}(n)$. 

\begin{corollary}
\label{general-modp-cor}
Let $X, \Delta, r, p$ be as above.  Then for all $k \geq 1$ there exists $\alpha_k > 0$ such that 
$$\#\{n \leq x: a_{\Delta, r}(n) \not \equiv 0 \pmod{p^k}\} = O \pa{\frac{x}{(\log x)^{\alpha_k}}}.$$
In particular, if we let $$\pi_{\Delta, r}(x; p^k) := \#\{n \leq x: a_{\Delta, r}(n) \equiv 0 \pmod{p^k}\},$$
then $$\lim_{x \to \infty}\frac{\pi_{\Delta, r}(x; p^k)}{x} = 1.$$

\end{corollary}

\begin{remark}
Corollary~\ref{general-modp-cor} also applies to any constant multiple of $f_{\Delta, r}$ with integral coefficients.  In the example below, we consider the coefficients of $$\frac{f_{-7,1}(\tau)}{90\sqrt{-7}} = q + O(q^2).$$  However, it is not always the case that the analogous normalization has integral coefficients.  
\end{remark}
 
\begin{example}
We illustrate Corollary~\ref{general-modp-cor} for $X = A_{1}^{24}$, $\Delta = -7$, $r = 1$.  Note that this is the same case considered in Example~\ref{borcherds-example}.  The first few coefficients of the normalized logarithmic derivative are given by 
\[\frac{f_{-7,1}(\tau)}{90\sqrt{-7}}=:\sum_{n\geq1}a_{-7, 1}(n)q^n=q+q^2-4371q^3+q^4+17773755q^5+\ldots\]

The prime $p = 2$ is split in $\Q(\sqrt{-7})$, and so Theorem~\ref{padic} and Corollary~\ref{general-modp-cor} do not apply.  Therefore, we expect the coefficients $a_{-7, 1}(n)$ to be equally distributed modulo 2, but cannot prove anything about them.  The prime $p = 3$ is inert, so Corollary~\ref{general-modp-cor} tell us that, asymptotically, 100\% of the coefficients $a_{-7, -1}(n)$ are divisible by 3.   We illustrate this behavior in Table~\ref{PadicNumerics}.  
\begin{table}
\label{PadicNumerics}
\caption{Divisibility of $a_{-7, 1}(n)$ by $p=2,3$}
\begin{tabular}{|c|c|c|}
\hline
$x$&$\pi_2(x)/x$&$\pi_3(x)/x$\\
\hline
50&0.38&0.64\\
\hline
100&0.45&0.68\\
\hline
150&0.47&0.69\\
\hline
200&0.49&0.71\\
\hline
250&0.48&0.71\\
\hline
300&0.49&0.72\\
\hline
\vdots & \vdots & \vdots \\
\hline
$\infty$ & .5? & 1 \\
\hline
\end{tabular}
\end{table}
\end{example}

\section{Umbral Moonshine}
\label{sec:umbral}
In this section, we summarize the main objects and conjectures of umbral moonshine.  However, we first briefly describe Mathieu moonshine, which umbral moonshine generalized.  

\subsection{Mathieu Moonshine}

In 2010, the study of a new form of moonshine commenced, called Mathieu moonshine.  Let $\mu(z, \tau) := \mu(z, z, \tau)$ be Zwegers' famous function from his thesis \cite{Zwegers:2008wk}, which is defined in the appendix.  Let $H^{(2)}(\tau)$ be the $q$-series  
$$H^{(2)}(\tau) := -8 \sum_{\omega \in \{\frac{1}{2}, \frac{1 + \tau}{2}, \frac{\tau}{2}\}} \mu(\omega, \tau) = 2 q^{-1/8}(-1 + 45 q + 231 q^2 + \dots),$$
which occurs in the decomposition of the elliptic genus of a K3 surface into irreducible characters of the $N = 4$ superconformal algebra.  This is a mock-modular form, and plays the role of $J(\tau)$ in Mathieu moonshine.  Eguchi, Ooguri, and Tachikawa conjectured that the Fourier coefficients encode dimensions of irreducible representations of the Mathieu group $M_{24}$ \cite{Eguchi:2011ux}.   This was extended to the full Mathieu moonshine conjecture by \cite{Cheng:2010va,Eguchi:2011uo,Gaberdiel:2010wu, Gaberdiel:2010wx}, which included providing mock modular forms $H_g^{(2)}$ for every $g \in M_{24}$.  The existence of an infinite dimensional $M_{24}$ module underlying the mock modular forms was shown by Gannon in 2012 \cite{Gannon:2012uv}.

\subsection{The Objects of Umbral Moonshine}

Cheng, Duncan, and Harvey generalized even further - conjecturing that Mathieu moonshine is but one example of a more general phenomenon which they call umbral moonshine \cite{Cheng:2013vr}.

For each of the 23 Niemeier root systems $X$, which are unions of irreducible simply-laced root systems with the same Coxeter number, they associate many objects, including a group $G^X$ (playing the role of $M$), a mock modular form $H^X(\tau)$ (playing the role of $j(\tau)$), and an infinite dimensional graded $G^X$ module $K^X$ (playing the role of the $M$-module $V^{\natural}$) Table~\ref{umbral_objects} gives a more complete list of the associated objects. 
\begin{table}[h!]
\caption{This table gives the objects associated to a Niemeier root system $X$}
\arraycolsep=10pt\def\arraystretch{1.5}
\begin{tabular}{| c |  l |}
\hline
$L^X$ & The Niemeier lattice corresponding to $X$ \\
\hline
$m$ & The Coxeter number of all irreducible components of $X$ \\
\hline
$W^X$ & The Weyl group of $X$ \\
\hline
$G^X := \Aut(L^X)/W^X$ & The umbral group corresponding to $X$ \\
\hline
$\pi^X$ & The (formal) product of Frame shapes of Coxeter elements of \\
& irreducible components of $X$ \\
\hline
$\Gamma^X$ & The genus zero subgroup attached to $X$ \\
\hline
$T^X$ & The normalized Hauptmodul of $\Gamma^X$, whose eta-product expansion \\ & corresponds to $\pi^X$ \\
\hline
$\ell$ & The lambency.  A symbol that encodes the genus zero group $\Gamma^X$.  \\& Sometimes used instead of $X$ to denote which case of umbral moonshine \\
& is being considered. \\
\hline
$\psi^X$ & The unique meromorphic Jacobi form of weight 1 and index $m$ \\
& satisfying certain conditions.\\
\hline
$H^X$ & The vector-valued mock modular form of weight 1/2 whose $2m$ components \\
&  furnish the theta expansion of the finite part of $\psi^X$.   \\
& Called the umbral mock modular form. \\
\hline
$S^X$ & The vector-valued cusp form of weight 3/2 which is the shadow of $H^X$.   \\
& Called the umbral shadow. \\
\hline
$H_g^X$ & The umbral McKay-Thompson series attached to $g \in G^X$.  \\
& It is a vector-valued mock modular form of weight 1/2, \\
& and equals $H^X$ when $g$ is the identity. \\
\hline
$S_g^X$ & The vector-valued cusp form {\it conjectured} to be the shadow of $H_g^X$.  \\
\hline
$K^X$ & The {\it conjectural} infinite dimensional graded $G^X$-module whose \\
&  graded super-dimension is encoded by $H^X$. \\
\hline
\end{tabular}
\label{umbral_objects}
\end{table}

The ADE classification of simply laced Dynkin diagrams allows us to classify the irreducible components of the Niemeier root systems $X$.  We will focus on the simplest cases - the root systems of pure $A$-type, i.e. $X = A_{m - 1}^{24/(m - 1)}$, where $(m - 1) \mid 24$. In these cases, the lambency $\ell$ is an integer and equals $m$, and $\Gamma^X = \Gamma_0(m)$. The case $X = A_1^{24}$ corresponds to Mathieu moonshine, with $G^X = M_{24}$ and $H^X = H^{(2)}$, as defined above.  We will generally refer to $H^X$ , $S^X$, $\psi^X$, and $T^X$ as $H^{(m)}$, $S^{(m)}$, $\psi^{(m)}$, and $j_m$ respectively.  These are the main quantities from Table~\ref{umbral_objects} that we will work with, and we will only define them for pure $A$-type.  This is done in Section~\ref{sec:definitions}.

\subsection{The Conjectures of Umbral Moonshine}
The main conjectures of umbral moonshine are as follows: 
\begin{enumerate}
\item The mock modular form $H^X$ encodes the graded super-dimension of a certain infinite-dimensional, $\Z/2m \Z \times \Q$-graded $G^X$-module $K^X$.
\item The graded super-characters $H_g^X$ arising from the action of $G^X$ on $K^X$ are vector-valued mock modular forms with concretely specified shadows $S_g^X$. 
\item The umbral McKay-Thompson series $H_g^X$ are uniquely determined by an optimal growth property which is directly analogous to the genus zero property of monstrous moonshine.  
\end{enumerate}

\section{vector-valued Modular Forms}
In this section, we follow \cite{Bruinier:2010ff} in giving the needed background on vector-valued modular forms, though we state results in less generality.  

\subsection{A Lattice related to $\Gamma_0(m)$}
We will define a lattice $L$ and a dual lattice $L'$ related to $\Gamma_0(m)$ such that the components of our vector-valued modular forms are labeled by the elements of $L'/L$.  

We consider the quadratic space $$V:= \{X \in \Mat_2(\Q): \tr(X) = 0\}$$
with the quadratic form $P(X) := m \det(X)$.\footnote{Note that this corrects a typo in \cite{Bruinier:2010ffz}.} The corresponding bilinear form is then $(X, Y) := -m \tr(XY)$.  
Let $L$ be the lattice
\[
 L:=\left\{ \begin{pmatrix} b& -a/m \\ c&-b \end{pmatrix}; \quad a,b,c\in\Z \right\}.
\]
The dual lattice is then given by
\[
 L':=\left\{ \begin{pmatrix} b/2m& -a/m \\ c&-b/2m \end{pmatrix}; \quad a,b,c\in\Z \right\}.
\]
We will switch between viewing elements of $L'$ as matrices and as quadratic forms, with the matrix $$X = \begin{pmatrix} b/2m& -a/m \\ c&-b/2m \end{pmatrix}$$ corresponding to the integral binary quadratic form $$Q = [mc, b, a] = mcx^2 + bxy + c y^2.$$ Note that then $P(X) = -\Disc(Q)/4m$.  

We identify $L'/L$ with $(\frac{1}{2m}\Z)/\Z$, and the quadratic form $P$ with the quadratic form $\frac{h}{2m} \mapsto \frac{-h^2}{4m}$ on $\Q/\Z$.  We will also occasionally identify $\frac{h}{2m} \in \Q/\Z$ with $h \in \Z/2m\Z$.  

For a fundamental discriminant $D$ and $r/2m \in L'/L$ with $r^2 \equiv D \pmod{4N}$, let  
\begin{equation}
\label{QDR}
Q_{D, r} := \{Q = [mc, b, a]: a, b, c \in \Z, \Disc(Q) = D, b \equiv r \pmod{2m}\}.
\end{equation}
 The action of $\Gamma_0(m)$ on this set is given by the usual action of congruence subgroups on binary quadratic forms.  We will later be working with $Q_{D, r}/\Gamma_0(m)$.  

\subsection{The Weil representation}
By $\Mp_2(\Z)$ we denote the integral metaplectic group. It consists of pairs
$(\gamma, \phi)$, where $\gamma = {\smallabcd \in \Sl_2(\Z)}$ and $\phi:\h\rightarrow \C$
is a holomorphic function with $\phi^2(\tau)=c\tau+d$.
The group $\widetilde{\G}:=\Mp_2(\Z)$ is generated by $S:=(\smallSmatrix,\sqrt{\tau})$ and $T:=(\smallTmatrix, 1)$. 

We consider the \textit{Weil representation} $\rho_L$ of $\Mp_2(\Z)$ 
corresponding to the discriminant form $L'/L$. We denote the standard basis elements of $\C[L'/L]$ by $\e_h$, $h/2m \in L'/L$.  Then the Weil representation $\rho_L$ associated with the discriminant form $L'/L$ is the unitary representation of $\widetilde{\Gamma}$ on $\C[L'/L]$ defined by 

\begin{equation*}
 \rho_L(T) \e_h = e(h^2/4m) \e_h,
\end{equation*}
and
\begin{equation*}
 \rho_L(S) \e_h = \frac{e(-1/8)}{\sqrt{2m}}
				\sum_{h' \in \Z/2m\Z} e(hh'/2m)\e_{h'}.
\end{equation*}

\subsection{Harmonic weak Maass forms}
If $f\colon \H \to \C[L'/L]$ is a function, we write $$f = \sum_{h \in \Z/2m\Z} f_h \e_h$$ for its decomposition into components.  For $k \in \frac{1}{2} \Z$, let $M_{k, \rho_L}^!$ denote the space of $\C[L'/L]$ valued weakly holomorphic modular forms of weight $k$ and type $\rho_L$ for the group $\widetilde{\Gamma}$.  The subspaces of holomorphic modular forms (resp. cusp forms) are denoted by $M_{k, \rho_L}$ (resp. $S_{k, \rho_L}$).  Now, assume that $k \leq 1$.  A twice continuously differentiable function $f: \H \to \C[L'/L]$ is called a \textit{harmonic weak Maass form} (of weight $k$ with respect to $\widetilde{\Gamma}$  and $\rho_L$) if it satisfies: 
\begin{enumerate}
\item $f(M\tau) = \phi(\tau)^{2k} \rho_L(M, \phi) f(\tau)$ for all $(M, \phi) \in \widetilde{\Gamma}$; 
\item $\Delta_k f = 0$;
\item There is a polynomial $$P_f(\tau) = \sum_{h \in \Z/2m\Z} \sum_{\substack{n \in \Z - \frac{h^2}{4m}, \\ -\infty << n \leq 0}} c^+(n, h)e(n\tau) \e_h$$ such that $$f(\tau)-P_f=O(e^{-\epsilon v})$$  for some $\epsilon >0$ as $v\to +\infty$. 
\end{enumerate}

Note here that $$\Delta_k := -v^2\pa{\frac{\partial^2}{\partial u^2} + \frac{\partial^2}{\partial v^2}} + i k v \pa{\frac{\partial}{\partial u} + i \frac{\partial}{\partial v}}$$ is the usual weight $k$ hyperbolic Laplace operator, and that $\tau = u + i v$. We denote the vector space of these harmonic weak Maass forms by $\calH_{k, \rho_L}$.   The Fourier expansion of any $f \in \calH_{k, \rho_L}$ gives a unique decomposition $f = f^+ + f^-$, where
\begin{eqnarray}
f^+(\tau) &=& \sum_{h \in \Z/2m\Z} \sum_{\substack{n \in \Z - \frac{h^2}{4m}, \\  -\infty << n}} c^+(n, h)e(n\tau) \e_h, \\
f^-(\tau) &=& \sum_{h \in L'/L} \sum_{\substack{n \in \Q, \\ n < 0}} c^-(n, h)W(2 \pi n v) e(n \tau) \e_h, 
\end{eqnarray}
and $W(x) := \int_{-2x}^\infty e^{-t}t^{-k} dt = \Gamma(1 - k, 2|x|)$ for $x < 0$.  Then $f^+$ is called the \textit{holomorphic part} and $f^-$ the \textit{nonholomorphic part} of $f$. The polynomial $P_f$ is also uniquely determined by $f$ and is called its \textit{principal part}. We define a \textit{mock modular form} of weight $k$ to be the holomorphic part $f^+$ of a harmonic weak Maass form $f$ of weight $k$ which has $f^- \neq 0$.  Its weight is just the weight of the harmonic weak Maass form.

Recall that there is an antilinear differential operator defined by 
$$\xi_k: \calH_{k, \overline{\rho}_L} \to S_{2 - k, \rho_L}, \; \; f(\tau) \mapsto \xi_k(f)(\tau) := 2iy^k\overline{\frac{\partial}{\partial\overline{\tau}}},$$
where $\overline{\rho}_L$ is the complex conjugate representation. 
The Fourier expansion of $\xi_k(f)$ is given by 
$$\xi_k(f) = -\sum_{h \in \Z/2m\Z} \sum_{n \in \Q, n > 0}(4 \pi n)^{1 - k} \overline{c^-(-n, h)} q^n \e_h.$$
The kernel of $\xi_k$ is equal to $M^!_{k, \overline{\rho}_L}$, and we have the following exact sequence:
$$0 \to M^!_{k, \overline{\rho}_L} \to \calH_{k, \overline{\rho}_L} \to S_{2 - k, \rho_L} \to 0.$$
We call $\xi_k(f)$ the \textit{shadow} of $f$.  Note that $\xi_k(f)$ uniquely determines $f^-$, but the $f^+$ is only determined up to the addition of a weakly holomorphic modular form.  

\section{Defining the Umbral mock modular Forms}
\label{sec:definitions}
In this section we define the mock modular forms $H^{(m)}$ from umbral moonshine, as well as their shadows $S^{(m)}$ and non-holomorphic parts.  Note that we only give definitions for the pure $A$-type cases - see \cite{Cheng:2013vr} for a more detailed and general definition.  We also refer the reader to the appendix for definitions of $\varphi_1^{(m)}(\tau, z), \mu_{m, 0}(\tau, z), \theta_{m, r}(\tau, z), \text{ and } R(u; \tau)$. 

For each lambency $m \in \{2, 3, 4, 5, 7, 9, 13, 25\}$, which correspond to the pure $A$-type cases, define the Jacobi form $\psi^{(m)}$ by 
$$\psi^{(m)}(\tau, z) := c_m \varphi_1^{(m)}(\tau, z) \mu_{1, 0}(\tau, z),$$ 
where $c_m = 2$ for $m = 2, 3, 4, 5, 7, 13$ and $c_m = 1$ for $m = 9, 25$.  We can break up $\psi^{(m)}$ into a finite part $\psi_F^{(m)}$ and a polar part $\psi_P^{(m)}$.  The polar part is given by 
$$\psi_P^{(m)}(\tau, z) = \frac{24}{m - 1} \mu_{m, 0}(\tau, z).$$
Then the mock modular form $H^{(m)}$ is defined by 
\begin{equation}
\psi_F^{(m)}(\tau, z) = \psi^{(m)}(\tau, z) - \psi_P^{(m)}(\tau, z) = \sum_{h \in \Z/2m\Z} H_h^{(m)}(\tau) \theta_{m, h}(\tau, z),
\end{equation}
where $$\theta_{m, h}(\tau, z) := \sum_{n \equiv h \pmod{2m}} q^{n^2/4m} y^k.$$  Note that $\psi^{(m)}$ satisfies an optimal growth condition, which is that 
\begin{equation}
\label{optimal}
q^{1/4m} H_h^X(\tau) = O(1)
\end{equation}
 as $\tau \to i \infty$ for all $h \in \Z/2m\Z$. 

We also define the shadow $S^{(m)}(\tau)$, the non-holomorphic part $F_r^{(m)}(\tau)$, and the harmonic weak Maass form $\widehat{H}^{(m)}(\tau)$ corresponding to the mock modular form $H^{(m)}$ via their components:
\begin{align} S_h^{(m)}(\tau) &:= \sum_{n \equiv h \pmod{2m}} n q^{n^2/4m},  \\
F_h^{(m)}(\tau) &:= \int_{-\overline{\tau}}^{i \infty} \frac{S_h^{(m)}(z)}{\sqrt{-i(z + \tau)}} dz \\
& =  -2 m q^{-(h-m)^2/4m} R\left(\frac{h - m}{2m} (2m\tau) + \frac{1}{2}; 2m\tau \right), \text{ and} \nonumber\\
\widehat{H}_h^{(m)}(\tau) &:= H_h^{(m)}(\tau) + F_h^{(m)}(\tau)
\end{align}

Note that by definition, $S_h^{(m)}(\tau) = -S_{-h}^{(m)}(\tau)$.  Therefore, $S_0^{(m)} = S_m^{(m)} = 0$.  The same is true of $H_h^{(m)}$.  We can write this in terms of Shimura's theta functions as $S_{h}^{(m)}(\tau) = \theta(\tau; h, 2m, 2m, x)$ \cite{Shimura:1973uk}.  Then using the transformation laws for his $\theta$-functions, we get that $S^{(m)}$ transforms as follows:
\begin{align*}
S^{(m)}_h(\tau + 1) &= e(h^2/4m)S^{(m)}_h(\tau), \text{ and}\\
S^{(m)}_h(-1/\tau) &= \tau^{3/2} \frac{e(-1/8)}{\sqrt{2m}} \sum_{k \pmod{2m}} e(kh/2m) S^{(m)}_k(\tau).
\end{align*}
Thus, we have
\begin{align*}
S^{(m)}(\tau + 1) &= \rho_L(T) S^{(m)}(\tau), \text{ and} \\
S^{(m)}(-1/\tau) &= \tau^{3/2} \rho_L(S) S^{(m)}(\tau).
\end{align*}

From these transformations, we see that $S^{(m)}(\tau): \H \to \C[L'/L]$ is a weight 3/2 vector-valued modular form transforming under the Weil representation $\rho_L$, i.e. an element of the space $M_{3/2, \rho_L}$.  From  \cite{Cheng:2013vr}, we know that $H^{(m)}$ is a mock modular form with shadow $S^{(m)}$.\footnote{In fact, it is the only vector-valued mock modular form with shadow $S^{(m)}$ satisfying the optimal growth condition in \ref{optimal}.} This gives us the following theorem. 

\begin{theorem}
We have that $\widehat{H}^{(m)}(\tau): \H \to \C[L'/L]$ is a weight 1/2 vector-valued harmonic weak Maass form transforming under the Weil representation $\overline{\rho}_L$, i.e., it is an element of $\calH_{1/2, \overline{\rho}_L}$. Moreover, it has shadow $S^{(m)}(\tau)$, non-holomorphic part $F^{(m)}$, and principal part $P(\tau) = - 2 q^{-1/4m} (\e_1 - \e_{2m - 1})$.
\end{theorem}

The reason we focus on the lattices of pure $A$-type is because this theorem is not true for the other cases - the vector-valued harmonic weak Maass forms no longer transform under the Weil representation.

\section{Relating umbral and monstrous moonshine}
\label{relating}
In this section, we explain the relationship between the mock modular forms $H^{(m)}$ from umbral moonshine and the Hauptmoduln $T_g$ from monstrous moonshine.  
\subsection{Twisted Generalized Borcherds Products}
\label{subsec:borcherds}
We begin by giving the theorem of Bruinier and Ono we will use.  

Let $c^+(n, h)$ be the $n$-th Fourier coefficient of $H_h^{(m)}$.  Let $(\Delta, r)$ be an admissible pair, so that $\Delta$ is a negative fundamental discriminant and $r^2 \equiv \Delta \pmod{4m}$.  Let $\Psi_{\Delta, r}(\tau, \widehat{H}^{m})$ be the twisted generalized Borcherds product defined in Theorem~\ref{mainthm}.  
\begin{theorem}(Theorem 6.1 in \cite{Bruinier:2010ff})
\label{Bruinier-thm}
We have that $\Psi_{\Delta, r}(\tau, \widehat{H}^{(m)})$ is a weight 0 meromorphic modular function on $\Gamma_0(m)$ with divisor $Z_{\Delta, r}(\widehat{H}^{(m)}).$\end{theorem}

For this theorem to make sense, we need to define the twisted Heegner divisor $Z_{\Delta, r}(\widehat{H}^{(m)})$ associated to $\widehat{H}^{(m)}$.  It is defined by
$$Z_{\Delta, r}(\widehat{H}^{(m)}) := \sum_{h \in \Z/2m\Z} \sum_{n < 0} c^+(n, h) Z_{\Delta, r}(n, h).$$
Since the principal part of $\widehat{H}^{(m)}$ is $- 2 q^{-h^2/4m} (\e_1 - \e_{2m - 1})$, this means that 
$$Z_{\Delta, r}(\widehat{H}^{(m)}) = 2 Z_{\Delta, r}\pa{\frac{-1}{4m}, \frac{-1}{2m}} - 2 Z_{\Delta, r}\pa{\frac{-1}{4m}, \frac{1}{2m}}.$$ 
Now, we just have to compute the divisors $Z_{\Delta, r}\pa{\frac{-1}{4m}, \frac{h}{2m}}$.  They are defined as follows.
$$Z_{\Delta, r}\left(\frac{-1}{4m}, \frac{h}{2m}\right) := \sum_{Q \in Q_{\Delta, hr}/\Gamma_0(m)} \frac{\chi_{\Delta}(Q)}{w(Q)} \alpha_Q, $$
where $w(Q) = 2$ for $\Delta < -4$, $\chi_\Delta$ is the generalized genus character defined in Gross-Kohnen-Zagier, and $\alpha_Q$ is the unique root of $Q(x, 1)$ in $\H$.  

\subsection{Proofs of Theorem~\ref{mainthm} and Corollary~\ref{twisted_trace}}
\begin{proof}[Proof of Theorem~\ref{mainthm}]
Theorem~\ref{Bruinier-thm} gives us that $\Psi_{\Delta, r}(\tau, \widehat{H}^{(m)})$ is a weight 0 meromorphic modular function on $\Gamma_0(m)$ with specified divisor, which is a discriminant $\Delta$ Heegner divisor.  For all of our $m$, $\Gamma_0(m)$ has genus zero.  Therefore, $\Psi_{\Delta, r}(\tau, \widehat{H}^{(m)})$ is a rational function in the Hauptmodul for $\Gamma_0(m)$.  The normalized Hauptmodul, which we call $j_m(\tau)$, is defined by 
\begin{equation}
j_m(\tau) := \frac{\eta(\tau)^{24/(m - 1)}}{\eta(m\tau)^{24/(m - 1)}} + \frac{24}{m - 1}.
\end{equation}
But using Table~\ref{intro-umbral}, we see that $j_m(\tau)$ is equal to $T_{g(X)}(\tau)$, the graded trace of $g(X) \in M$ on $V$.  
\end{proof}

\begin{proof}[Proof of Corollary~\ref{twisted_trace}]
From Theorem~\ref{mainthm}, we have that  $$\prod_{n  = 1}^\infty P_\Delta(q^n)^{c^+\pa{\frac{|\Delta| n^2}{4 m}, \frac{rn}{2m}}} = \prod_i (T_g(\tau) - T_g(\alpha_i))^{\gamma_i}.$$ We equate the $q^1$ Fourier coefficients of each side, using Table~\ref{monster} to get the Fourier expansion $$T_g(\tau) = \frac{1}{q}  + O(q).$$ 
\end{proof}

\subsection{Examples}
For each pure A-type case $X$ with coxeter number $m$, we illustrate how to write $\Psi_{\Delta, r}(\tau, \widehat{H}^{(m)})$ as a rational function in $j_m$.  Note that here $\Delta < 0$ is a fundamental discriminant and $r \in \Z$ is such that $\Delta \equiv r^2 \pmod{4m}$.  

First we work out an example for $m = 2$ in some detail, then list one example for each $m$.  In Section~\ref{subsec:quadratic_forms}, we explain how to find representatives of $Q_{\Delta, r}/\Gamma_0(m)$ using a method of Gross, Kohen, and Zagier.  

Consider the case $m = 2, \Delta = -7, r = 1$.  Using the method of Section~\ref{subsec:quadratic_forms}, we compute that $Q_{-7, 1}/\Gamma_0(2) = \{Q_1, Q_2\}$ and that $Q_{-7, -1}/\Gamma_0(2) = \{-Q_1, -Q_2\}$, where the quadratic forms $Q$, their Heenger points $\alpha_Q$, and their generalized genus characters $\chi_\Delta(Q)$ are given in Table~\ref{quadratic_form_example}.  We also include the value of $j_2$ at each Heegner point.  
\begin{table}[h!]
\caption{Quadratic forms needed for $m = 2$, $\Delta = -7, r = 1$ case}
\[
\arraycolsep=10pt\def\arraystretch{2.2}
\begin{array}{| c | c | c | c |}
\hline 
\text{quadratic form} = Q & \alpha_Q &  \chi_{\Delta}(Q) & j_2(\alpha_Q) \\
\hline \hline
Q_1 = [ 2, 1, 1] & \alpha_1 = \frac{-1 + \sqrt{-7}}{4} & 1 &\gamma_1 := \frac{1 + 45 \sqrt{-7}}{2}  \\
\hline
Q_2 = [-2, 1, -1]& \alpha_2 = \frac{1 + \sqrt{-7}}{4} & - 1 & \gamma_2 := \frac{1 - 45 \sqrt{-7}}{2} \\
\hline
-Q_2 & \alpha_2 & 1 & \gamma_2\\
\hline
-Q_1  & \alpha_1 & - 1& \gamma_1 \\
\hline
\end{array}
\]
\label{quadratic_form_example}
\end{table}
Using the table, the divisor of $\Psi_{-7, 1}(\tau)$ is given by: 
$$(-\alpha_1 + \alpha_2) -(\alpha_1 - \alpha_2) = 2 \alpha_2 - 2 \alpha_1.$$
Therefore, 
$$\Psi_{-7, 1}(\tau, \widehat{H}^{(2)}) = \frac{(j_2(\tau) - \gamma_2)^2}{(j_2(\tau) - \gamma_1)^2}.$$

Similarly, for each value of $m$ corresponding to a pure A-type case, we demonstrate in Table~\ref{borcherd_examples} how to write $\Psi_{\Delta, r}(\tau, \widehat{H}^{(m)})$ as a rational function in $j_m$ for some nice choice of $\Delta, r$.  In all the examples we consider, $$\Psi_{\Delta, r}(\tau, \widehat{H}^{(m)}) =\frac{(j_m(\tau) - \gamma_2)^2}{(j_m(\tau) - \gamma_1)^2}$$ for some $\gamma_1, \gamma_2 \in \calO_{\Q(\sqrt{\Delta})}$. Note that $\Psi_{\Delta, r}$ will not always be a rational function of this particular form - we always picked $\Delta$ with class number 1.  

\begin{table}
\caption{Examples}
\label{borcherd_examples}
\[
\arraycolsep=10pt\def\arraystretch{2.2}
\begin{array}{| c | c | c | c | c |}
\hline
m & \Delta & r & \gamma_1 & \gamma_2 \\
\hline \hline
2 & -7 & 1 &  \frac{1 + 45 \sqrt{-7}}{2} & \frac{1 - 45 \sqrt{-7}}{2} \\
\hline
3 & -11 & 1 & 17 + 8 \sqrt{-11} & 17- 8 \sqrt{-11} \\
\hline
4 & -7 & 3 & \frac{-15 +3\sqrt{-7}}{2} & \frac{-15 - 3 \sqrt{-7}}{2} \\
\hline
5 & -11 & 3 & -3 + 2 \sqrt{-11}& -3 - 2 \sqrt{-11} \\
\hline
7 & -19 & 3 & \frac{3 + 3 \sqrt{-19}}{2}& \frac{3 - 3 \sqrt{-19}}{2}\\
\hline
9 & -11 & 5 & -1 + \sqrt{-11} & -1 - \sqrt{-11} \\ 
\hline
13 & -43 & 3 & \frac{7 +  \sqrt{-43}}{2} & \frac{7 -  \sqrt{-43}}{2} \\
\hline
25 & -19 & 9 & \frac{\sqrt{-19}}{2} & \frac{- \sqrt{-19}}{2} \\
\hline
\end{array}
\]
\end{table}

\subsection{Computing the elements in $Q_{\Delta, r}/\Gamma_0(m)$}
\label{subsec:quadratic_forms}

In this section, we explain how to compute $Q_{\Delta, r}/\Gamma_0(m)$, following \cite{Gross:1987ul}.  

Let $Q_{\Delta, r}^0$ be the subset of primitive forms. Then we have a $\Gamma_0(m)$-invariant bijection of sets 
$$Q_{\Delta, r} = \bigcup_{\ell^2 \mid \Delta} \left( \bigcup_{h \in S(\ell)} \ell Q_{\Delta/\ell^2, h}^0\right),$$
where $S(\ell) := \{h \in \Z/2m \Z: h^2 \equiv \Delta/\ell^2 \pmod{4m}, \ell h \equiv r \pmod{2m}\}.$ 
Since we pick $\Delta$ to be a fundamental discriminant, the only possible prime we need to worry about is $\ell = 2$.  In our examples, we always choose $\Delta, r$ such that $S(2) = \emptyset$.  In this case, we just need to work with $Q_{\Delta, r}^0$.  

Now, let $n := \left(m, r, \frac{r^2 - \Delta}{4m}\right)$.  Then for $Q = [mc, b, a] \in Q_{\Delta, r}^0$, define $n_1 := (m, b, a), n_2 :=(m, b, c)$, which are coprime and have product $n$.  We have the following result: 

\begin{lemma}(Section 1.1 of \cite{Gross:1987ul})
Define $n$ as above and fix a decomposition $n = n_1 n_2$ with $n_1, n_2$ positive and relatively prime.  Then there is a 1:1 correspondence between the $\Gamma_0(m)$-equivalence classes of forms $[cm, b, a] \in Q_{\Delta, r}^0$ satisfying $(m, b, a) = n_1, (m, b, c) = n_2$ and the $\SL_2(\Z)$ equivalence classes of forms in $Q_{\Delta}^0$ given by $Q = [mc, b, a] \mapsto \tilde{Q} = [c m_1, b, a m_2]$, where $m_1 \cdot m_2$ is any decomposition of $m$ into coprime positive factors satisfying $(n_1, m_2) = (n_2, m_1) = 1$.  In particular, $|Q_{\Delta, r}^0/\Gamma_0(m)| = 2^v |Q_\Delta^0/\SL_2(\Z)|$, where $v$ is the number of prime factors of $n$.  
\end{lemma}

Note that $|Q_\Delta^0/\SL_2(\Z)|$ equals $2h(\Delta)$ for $\Delta < 0$, where the factor of 2 arises because $Q_\Delta^0$ also contains negative semi-definite forms.  

In our examples, we always choose $\Delta, r$ such that $n = 1$, so that $|Q_{\Delta, r}^0/\Gamma_0(m)| = |Q_\Delta^0/\SL_2(\Z)| = 2 h(\Delta)$, where $h(\Delta)$ is the class number of $\Q(\sqrt{\Delta})$.  The theory of reduced forms allows us to easily compute $Q_\Delta^0/\SL_2(\Z)$.  

\section{$p$-adic properties of the logarithmic derivative}
\subsection{$p$-adic modular forms}
\label{sec:padic-def}
For each $i \in \N$, let $f_i = \sum a_i(n) q^n$ be a modular form of weight $k_i$ with $a_i(n) \in \Q$.  If for each $n$, the $a_i(n)$ converge $p$-adically to $a(n) \in \Q_p$, then $f := \sum a(n)q^n$ is called a $p$-adic modular form.  For $p \neq 2$, we define the weight space $$W := \varprojlim_t \Z/\phi(p^t)\Z = \Z_p \times \Z/(p -1)\Z.$$  For $p = 2$, we define $$W := \varprojlim_t \Z/2^{t - 2}\Z = \Z_2.$$ Then the $k_i$ converge to an element $k \in W$, which we call the weight of $f$. We identify integers by their image in $\Z_p \times \{0\}$.  

\subsection{Proof of Theorem~\ref{mero-modform}}

\begin{proof}[Proof of Theorem~\ref{mero-modform}]
By Theorem \ref{mainthm}, $\Psi_{\Delta,r}(\tau)$ is a meromorphic modular function, so that $\Theta(\Psi_{\Delta, r}(\tau))$ is a weight 2 meromorphic modular form on $\Gamma_0(m)$. Thus, the logarithmic derivative $\frac{\Theta(\Psi_{\Delta, r}(\tau))}{\Psi_{\Delta, r}(\tau)}$ is a weight 2 meromorphic modular form on $\Gamma_0(m)$ whose poles are simple and are supported on Heegner points of discriminant $\Delta.$ 
\end{proof}

\subsection{Proof of Theorem~\ref{padic} and its corollary}

\begin{proof}[Proof of Theorem~\ref{padic}]

We show that if $(\Delta, r)$ is an admissible pair and $p$ is inert or ramified in $\Q(\sqrt{\Delta})$, that $$f_{\Delta, r} := \frac{\Theta(\Psi_{\Delta, r}(\tau))}{\Psi_{\Delta, r}(\tau)}$$
is a $p$-adic modular form of weight 2.  Say $f$ has poles at $\alpha_1, \dots, \alpha_n$, all of which are CM points of discriminant $\Delta$.  For each $\alpha_i$, there is some zero $\beta_i$ of $E_{p -1}$ such that $j(\tau) - j(\alpha_i) \equiv j(\tau) - j(\beta_i)$ (see Theorem 1 of \cite{Kaneko:1998wt}).  Then let $$\mathcal{E}:= E_{p - 1} \prod_{i} \frac{(j(\tau) - j(\alpha_i))}{(j(\tau) - j(\beta_i))}.$$  This has weight $p - 1$, is congruent to 1 modulo $p$, has zeros at $\alpha_1, \dots, \alpha_n$, and has no poles.  Let $f_t := f \mathcal{E}^{(p^t)}$.  Then $f_t \equiv f \pmod{p^t}$ and is a modular form of weight $k_t = 2 + (p - 1)p^t \equiv 2 \pmod{\phi(p^{t + 1})}$, so $f$ is a $p$-adic modular form of weight $2$.
\end{proof}

\begin{proof}[Proof of Corollary~\ref{general-modp-cor}]
This corollary follows directly for the coefficients of any $p$-adic modular form using the following beautiful result, proven by Serre \cite{Serre:1974tp} using the theory of Galois representations.
\begin{lemma}[Serre \cite{Serre:1974tp} Theorem 4.7 (I)]
Let $K$ be a number field and $\mathcal O_K$ the ring of integers of $K$. Suppose $f(\tau)=\sum_{n\geq0}a_nq^n\in\mathcal O_K[[q]]$ is a modular form of integer weight $k\geq1$ on a congruence subgroup. For any prime $p$, let $\mathfrak p$ be a prime above $p$ in $\mathcal O_K$. Let $m \geq 1$.  Then there exists a positive constant $\alpha_m$ such that 
\[\#\left\{n\leq X\colon a_n\not\equiv0\pmod{\mathfrak p}^m\right\}=O\left(\frac{X}{(\log X)^{\alpha_m}}\right).\]
\end{lemma}

\end{proof}

\section{Appendix: Definitions of jacobi forms, theta functions, etc.}
\label{sec:appendix}

We define the Jacobi theta functions $\theta_i(\tau, z)$ as follows for $q := e(\tau)$ and $y := e(z)$. 

\begin{eqnarray*}
\theta_2(\tau, z) &:=& q^{1/8}y^{1/2} \prod_{n = 1}^\infty (1 - q^n)(1 + y q^n)(1 + y^{-1}q^{n - 1}) \\
\theta_3(\tau, z) &:=& \prod_{n = 1}^\infty (1 - q^n)(1 + y q^{n - 1/2})(1 + y^{-1}q^{n - 1/2}) \\
\theta_4(\tau, z) &:=& \prod_{n = 1}^\infty (1 - q^n)(1 - y q^{n - 1/2})(1 - y^{-1}q^{n - 1/2}) \\
\end{eqnarray*}

We use them to define weight zero index $m - 1$ weak Jacobi forms $ \varphi_1^{(m)}$ as follows.  Let 
\begin{align*}
\varphi_1^{(2)} &:= 4 (f_2^2 + f_3^2 + f_4^2),\\
\varphi_1^{(3)} &:= 2(f_2^2 f_3^2 + f_3^2 f_4^2 + f_4^2 f_2^2), \\
\varphi_1^{(4)} &:= 4 f_2^2 f_3^2 f_4^2, \\
\varphi_1^{(5)} &:= \frac{1}{4} \pa{\varphi_1^{(4)} \varphi_1^{(2)} - (\varphi_1^{(3)})^2} \\
\varphi_1^{(7)} &:= \varphi_1^{(3)} \varphi_1^{(5)} - (\varphi_1^{(4)})^2 \\
\varphi_1^{(9)} &:= \varphi_1^{(3)} \varphi_1^{(7)} - ( \varphi_1^{(5)})^2 \\
\varphi_1^{(13)} &:= \varphi_1^{(5)} \varphi_1^{(9)} - 2 (\varphi_1^{(7)})^2 \\
\end{align*}
where $f_i(\tau, z) := \theta_i(\tau, z)/\theta_i(\tau, 0)$ for $i = 2, 3, 4$.   

For the remaining positive integers $m$ with $m \leq 25$, we define $\varphi_1^{(m)}$ recursively.  
\\
For $(12, m - 1) = 1$ and $m > 5$ we set
$$\varphi_1^{(m)} = (12, m - 5) \varphi_1^{(m - 4)} \varphi_1^{(5)} + (12, m - 3) \varphi_1^{(m - 2)} \varphi_1^{(3)} - 2(12, m - 4) \varphi_1^{(m - 3)} \varphi_1^{(4)}.$$
For $(12, m - 1) = 2$ and $m > 10$ we set
$$\varphi_1^{(m)} = \frac{1}{2}\pa{(12, m - 5) \varphi_1^{(m - 4)} \varphi_1^{(5)} + (12, m - 3) \varphi_1^{(m - 2)} \varphi_1^{(3)} - 2 (12, m - 4) \varphi_1^{(m - 3)} \varphi_1^{(4)}}.$$
For $(12, m - 1) = 3$ and $m > 9$, we set
$$ \varphi_1^{(m)} = \frac{2}{3}(12, m - 4) \varphi_1^{(m - 3)} \varphi_1^{(4)} + \frac{1}{3} (12, m - 7) \varphi_1^{(m - 6)} \varphi_1^{(7)} - (12, m - 5) \varphi_1^{(m - 4)} \varphi_1^{(5)}.$$
For $(12, m - 1) = 4$ and $m > 16$ we set 
$$ \varphi_1^{(m)} = \frac{1}{4} \pa{(12, m -13) \varphi_1^{(m - 12)} \varphi_1^{(13)} + (12, m - 5) \varphi_1^{(m - 4)} \varphi_1^{(5)} - (12, m - 9) \varphi_1^{(m - 8)} \varphi_1^{(9)}}.$$
For $(12, m - 1) = 6$ and $m > 18$ we set 
$$ \varphi_1^{(m)} = \frac{1}{3}(12, m - 4) \varphi_1^{(m - 3)} \varphi_1^{(4)} + \frac{1}{6} (12, m - 7) \varphi_1^{(m - 6)} \varphi_1^{(7)} - \frac{1}{2} (12, m - 5) \varphi_1^{(m - 4)} \varphi_1^{(5)}.$$
For $m = 25$, we set
$$ \varphi_1^{(25)} = \frac{1}{2} \varphi_1^{(21)} \varphi_1^{(5)} - \varphi_1^{(19)} \varphi_1^{(7)} + \frac{1}{2} ( \varphi_1^{(13)})^2.$$

See the appendix of \cite{Cheng:2013vr} for more information on the space of weight zero Jacobi forms.  

We use two versions of an Appell-Lerch sum.  The first is the generalized Appell-Lerch sum $\mu_{m, 0}$, defined as in \cite{Cheng:2013vr}.  It is given by
$$\mu_{m, 0}(\tau, z) := -\sum_{k \in \Z}q^{m k^2} y^{2 m k} \frac{1 + y q^k}{1 - y q^k},$$ and is the holomorphic part of a weight 1 index $m$ ``real-analytic Jacobi form''.  

Zwegers \cite{Zwegers:2008wk} uses a slightly different version of the Appell-Lerch sum.  He first defines  
the theta function $$\vartheta(z, \tau) := \sum_{\nu \in 1/2 + \Z} q^{\nu^2/2} y^{\nu} e(\nu/2).$$
Then he defines
$$\mu(u, v; \tau) := \frac{e(u/2)}{\vartheta(v; \tau)} \sum_{n \in \Z} \frac{(-1)^n q^{(n^2 + n)/2} e(n v)}{1 - q^n e(u)}.$$
This is completed to a ``real-analytic Jacobi form'' $\tilde{\mu}(u, v; \tau)$ of weight 1/2 by letting $$\tilde{\mu}(u, v; \tau) := \mu(u, v; \tau) + \frac{i}{2}R(u - v; \tau),$$
where $$R(z, \tau) := \sum_{\nu \in 1/2 + \Z} \left\{ \text{sgn}(\nu) - E(\nu + a) \sqrt{2t}\right\} (-1)^{\nu - 1/2} q^{-\nu^2/2} y^{-\nu},$$
$t := \Im(\tau)$, $a := \frac{\Im(u)}{\Im(\tau)}$, and $E(z) := 2 \int_0^z e^{-\pi u^2} du$. 

\bibliography{moonshine_paper}
\bibliographystyle{plain}

\end{document}